\documentclass[11pt,reqno,a4paper,]{article}
\usepackage{amssymb,amsmath}
\usepackage{amssymb,amsthm}
\usepackage[pdftex]{graphicx}
\newtheorem{Theorem}{Theorem}

\newtheorem{Corollary}{Corollary}
\newtheorem{Definition}{Definition}
\numberwithin{Definition}{section}
\numberwithin{Theorem}{section}
\numberwithin{Lemma}{section}

\numberwithin{Corollary}{section}
\numberwithin{Example}{section}
\numberwithin{Remark}{section}
\usepackage{enumerate}
\title{On semi-symmetric non-metric connexion}
\author{S. K. Chaubey\\
Department of Mathematics,
Faculty of Science,\\
Banaras Hindu University,
Varanasi-221005,
India.}
\date{}
\begin{document}
\maketitle
\footnotetext[1]{ E-mail: sk22$_{-}$math@yahoo.co.in\\}
\begin{abstract}
H. A. Hayden [1] introduced the idea of semi-symmetric non-metric connection on a Riemannian manifold in  (1932). Agashe and Chafle \cite{1} defined and studied semi-symmetric non-metric connection on a Riemannian manifold. In the present paper, we define a new type of semi-symmetric non-metric connexion in an almost contact metric manifold and studied its properties. In the end, we have studied some properties of the covariant almost analytic vector field equipped with semi-symmetric non-metric connection.
\end{abstract}
\textbf{2000 Mathematics Subject Classification : 53B15.}\\
\textit{\textbf{Keywords and Phrases :}} Semi-symmetric non-metric connection, Almost contact metric manifold, generalised co-symplectic manifold, generalised quasi-Sasakian manifold, Covariant almost analytic vector field.
\maketitle


\section{Preliminaries}If there are a tensor field $F$ of type $(1,1)$ , a vector field $T$, a 1-form $A$ and metric $g$, satisfying next equations for arbitrary vector fields $X,Y,Z\in{T_p}$ , ${p\in{M_n}}$
\begin{equation}
\label{1}
\overline{\overline{X}}+X=A(X)T ,
\end{equation}
\begin{equation}
\label{2}
A(\overline{X})=0 ,
\end{equation}
\begin{equation}
\label{3}
'F(X,Y){\stackrel{\mathrm{def}}{=}}g(\overline{X},Y) ,
\end{equation}
\begin{equation}
\label{4}
g(\overline{X},\overline{Y})=g(X,Y)-A(X)A(Y) ,
\end{equation}
\begin{equation*}
\overline{X}{\stackrel{\mathrm{def}}{=}}{FX} ,
\end{equation*}
then the n-dimensional differentiable manifold $M_n$ is called an almost contact metric manifold \cite{6}.
	
	An almost contact metric manifold satisfying
\begin{equation}
\label{5}
(D_X{'F})(Y,Z)=A(Y)(D_X{A})(\overline{Z})-A(Z)(D_X{A})(\overline{Y})
\end{equation}
\begin{eqnarray}
\label{6}
D_X{'F})(Y,Z)&+&(D_Y{'F})(Z,X)+(D_Z{'F})(X,Y)\nonumber\\&=&
A(X)[(D_Y{A})({\overline{Z}})-(D_Z{A})({\overline{Y}})]+A(Y)[(D_Z{A})({\overline{X}})\nonumber\\&-&
(D_X{A})({\overline{Z}})]+A(Z)[(D_X{A})({\overline{Y}})-(D_Y{A})({\overline{X}})]
\end{eqnarray}
are called generalized co-symplectic manifold and generalized quasi-Sasakian manifold respectively \cite{5}.
	
	The Nijenhuis tensor in generalized cosymplectic manifold is given by
\begin{eqnarray}
\label{7}
'N(X,Y,Z)&=&(D_{\overline{X}}{'F})(Y,Z)-(D_{\overline{Y}}{'F})(X,Z)\nonumber\\&+&
(D_X{'F})(Y,{\overline{Z}})-(D_Y{'F})(X,{\overline{Z}})
\end{eqnarray}
If on any manifold , $T$ satisfies
\begin{equation}
\label{8}
(D_X{A})(\overline{Y})=-(D_{\overline{X}}A)(Y)=(D_Y{A})(\overline{X})
\end{equation}
\begin{equation*}
\Leftrightarrow(D_X{A})(Y)=(D_{\overline{X}}A)({\overline{Y}})=-(D_Y{A})(X),
\end{equation*}
\begin{equation*}
D_T{F}=0,
\end{equation*}
then $T$ is said to be of first class and the manifold is said to be of first class \cite{5}.
If on an almost contact metric manifold $T$ satisfies
\begin{equation}
\label{9}
(D_X{A})(\overline{Y})=(D_{\overline{X}}A)(Y)=-(D_Y{A})(\overline{X})
\end{equation}
\begin{equation*}
\Leftrightarrow(D_X{A})(Y)=-(D_{\overline{X}}A)({\overline{Y}})=-(D_Y{A})(X),
\end{equation*}
\begin{equation*}
D_T{F}=0,
\end{equation*}
then $T$ is said to be of second class and the manifold is said to be of second class \cite{5}.
\section{Semi-symmetric non-metric connexion}
\begin{Definition}
Let $D$ be a Riemannian connexion , then we define an affine connexion $\tilde{B}$ as
\end{Definition}
\begin{equation}
\label{10}
\tilde{B}_X{Y}=D_X{Y}+'F(X,Y)T
\end{equation}
satisfying
\begin{equation}
\label{11}
\tilde{S}(X,Y)=2'F(X,Y)T
\end{equation}
and
\begin{equation}
\label{12}
(\tilde{B}_{X}g)(Y,Z)=-A(Y)'F(X,Z)-A(Z)'F(X,Y)
\end{equation}
is called a semi-symmetric non-metric connexion.
Also,
\begin{equation}
\label{13}
\tilde{S}(X,Y,Z){\stackrel{\mathrm{def}}{=}}g(\tilde{S}(X,Y),Z)=2A(Z)'F(X,Y)
\end{equation}
\begin{equation}
\label{14}
(\tilde{B}_{X}F)(Y)=(D_X{F})(Y)+g(\overline{X},\overline{Y})T
\end{equation}
\begin{equation}
\label{15}
(\tilde{B}_X{A})(Y)=(D_X{A})(Y)-g(\overline{X},Y)
\end{equation}
\begin{Theorem}
If $D$ be a Riemannian connexion and $\tilde{B}$ be a semi-symmetric non-metric connexion, then on an almost contact metric manifold $\tilde{S}$ is hybrid.
\end{Theorem}
\begin{Theorem}
If the manifold is of first class with respect to the Riemannian connexion $D$, then it is also first class with respect to the semi-symmetric non-metric connexion $\tilde{B}$
\end{Theorem}
\begin{proof}
From (\ref{15})
\begin{equation}
\label{16}
(D_{\overline{X}}A)(Y)=(\tilde{B}_{\overline{X}}A)(Y)+g(\overline{X},\overline{Y})
\end{equation}
and
\begin{equation}
\label{17}
(D_X{A})(\overline{Y})=(\tilde{B}_{X}A)(\overline{Y})-g(\overline{X},\overline{Y})
\end{equation}
Adding (\ref{16}) and (\ref{17}) and then using (\ref{4}), we have
\begin{equation}
\label{18}
(D_{\overline{X}}A)(Y)+(D_X{A})(\overline{Y})=(\tilde{B}_{\overline{X}}{A})(Y)+(\tilde{B}_X{A})(\overline{Y})
\end{equation}
In view of (\ref{8}), (\ref{18}) becomes
\begin{equation}
\label{19}
(\tilde{B}_{\overline{X}}A)(Y)=-(\tilde{B}_{X}A)(\overline{X})
\end{equation}
Similarly, we have from (\ref{15})
\begin{equation}
\label{20}
(\tilde{B}_X{A})(\overline{Y})=(\tilde{B}_Y{A})(\overline{X})
\end{equation}
Equations (\ref{19}) and (\ref{20}) give
\begin{equation}
\label{21}
(\tilde{B}_X{A})(\overline{Y})=-(\tilde{B}_{\overline{X}}{A})(Y)=(\tilde{B}_{Y}{A})(\overline{X})
\end{equation}
Now taking covariant derivative of $FY={\overline{Y}}$ with respect to $\tilde{B}$ and using (\ref{8}) and (\ref{10}), we obtain
\begin{equation*}
\tilde{B}_T{F}=0
\end{equation*}
Hence the theorem.
\end{proof}
\begin{Theorem}
Let $D$ be a Riemannian connexion and $\tilde{B}$ be a semi-symmetric non-metric connexion.Then an almost contact metric manifold is a generalized quasi-Sasakian manifold of the first kind if
\begin{equation*}
(\tilde{B}_X{F})(Y,Z)+(\tilde{B}_Y{F})(Z,X)+(\tilde{B}_Z{F})(X,Y)=0
\end{equation*}
\end{Theorem}
\begin{proof}
We have,
\begin{eqnarray*}
X('F(Y,Z))&=&(\tilde{B}_X{'F})(Y,Z)+'F(\tilde{B}_X{Y},Z)+'F(Y,\tilde{B}_X{Z})\nonumber\\&=&
(D_X{'F})(Y,Z)+'F(D_X{Y},Z)+'F(Y,D_X{Z})
\end{eqnarray*}
Then,
\begin{equation*}
(D_X{'F})(Y,Z)=(\tilde{B}_X{'F})(Y,Z)+'F(\tilde{B}_X{Y}-D_X{Y},Z)+'F(Y,\tilde{B}_X{Z}-D_X{Z})
\end{equation*}
Using (\ref{10}) and (\ref{3}),we get
\begin{equation}
\label{22}
(D_X{'F})(Y,Z)=(\tilde{B}_X{'F})(Y,Z)
\end{equation}
Taking covariant derivative of $A(\overline{Z})=0$ with respect to the Riemannian connexion $D$ and using (\ref{10}),we obtain
\begin{equation}
\label{23}
(D_X{A})(\overline{Z})=(\tilde{B}_X{A})(\overline{Z})+g(\overline{X},\overline{Z})
\end{equation}
Using (\ref{8}), (\ref{22}) and (\ref{23}), (\ref{6}) gives the required result.
\end{proof}
\begin{Theorem}
A quasi-Sasakian manifold be normal if and only if
\begin{eqnarray*}
(\tilde{B}_X{'F})(Y,Z)&=&A(Y)[(\tilde{B}_Z{A})(\overline{X})+g(X,Z)]\nonumber\\&+&
A(Z)[(\tilde{B}_{\overline{X}}{A})(Y)+g(X,Y)],
\end{eqnarray*}
where $\tilde{B}$ being semi-symmetric non-metric connexion.
\end{Theorem}
\begin{proof}
From (\ref{15}),
\begin{equation}
\label{24}
(D_{\overline{X}}A)(Y)=(\tilde{B}_{\overline{X}}{A})(Y)-g(\overline{X},\overline{Y})
\end{equation}
The necessary and sufficient condition that a quasi-Sasakian manifold be normal \cite{5} is
\begin{equation}
\label{25}
(D_X{'F})(Y,Z)=A(Y)(D_Z{A})(\overline{X})+A(Z)(D_{\overline{X}}A)(Y)
\end{equation}
Using (\ref{22}), (\ref{23}) and (\ref{24}) in (\ref{25}), we obtain the required result.
\end{proof}
\begin{Theorem}
An almost contact metric manifold to be generalized cosymplectic manifold if
\begin{eqnarray}
\label{26}
(\tilde{B}_X{'F})(Y,Z)&=&A(Y)[(\tilde{B}_X{A})(\overline{Z})+g(X,Z)]\nonumber\\&-&
A(Z)[(\tilde{B}_X{A})(\overline{Y})+g(X,Y)]
\end{eqnarray}
\end{Theorem}
\begin{proof}
From (\ref{15})
\begin{equation}
\label{27}
(D_X{A})(\overline{Z})=(\tilde{B}_X{A})(\overline{Z})+g(\overline{X},\overline{Z})
\end{equation}
Using (\ref{22}) and (\ref{27}) in (\ref{5}), we obtain the required result.
\end{proof}
\begin{Theorem}
On generalized co-symplectic manifold, $F$ is killing with respect to $\tilde{B}$ if and only if
\begin{equation*}
(\tilde{B}_X{A})(\overline{Z})+g(X,Z)=0
\end{equation*}
\end{Theorem}
\begin{proof}
Since $F$ is killing with respect to semi-symmetric non-metric connexion $\tilde{B}$, we have
\begin{equation}
\label{28}
(\tilde{B}_X{'F})(Y,Z)+(\tilde{B}_Y{'F})(X,Z)=0
\end{equation}
Using (\ref{26}), (\ref{28}) becomes
\begin{equation*}
A(X)[(\tilde{B}_Y{'F})(T,Z)]+A(Y)[(\tilde{B}_X{'F})(T,Z)]+2A(Z)g(X,Y)=0
\end{equation*}
or,
\begin{eqnarray}
\label{29}
2A(Z)g(X,Y)&+&A(X)[(\tilde{B}_Y{A})(\overline{Z})+g(\overline{Y},\overline{Z})]\nonumber\\&+&
A(Y)[(\tilde{B}_X{A})(\overline{Z})+g(\overline{X},\overline{Z})]=0
\end{eqnarray}
Putting $T$ for $X$ in this equation ,we obtain
\begin{equation}
\label{30}
A(X)(\tilde{B}_T{A})(\overline{Z})+(\tilde{B}_X{A})(\overline{Z})+g(\overline{X},\overline{Z})+2A(Z)A(X)=0
\end{equation}
Putting $T$ for $X$ in this equation , we get
\begin{equation}
\label{31}
(\tilde{B}_T{A})(\overline{Z})+A(Z)=0
\end{equation}
From (\ref{30}) and (\ref{31}), we get the result. Converse part is obvious.
\end{proof}
\begin{Theorem}
If $U$ is killing , then on generalized co-symplectic manifold
\begin{equation*}
'N(X,Y,Z)-d'F(X,Y,{\overline{Z}})=2A(Z)(\tilde{B}_{\overline{Y}}{A})(\overline{X})
\end{equation*}
\end{Theorem}
\begin{proof}
From (\ref{7}) and (\ref{22}), we get
\begin{eqnarray*}
'N(X,Y,Z)&-&d'F(X,Y,{\overline{Z}})=(\tilde{B}_{\overline{X}}{'F})(Y,Z)-(\tilde{B}_{\overline{Y}}{'F})(Y,Z)\nonumber\\&+&
(\tilde{B}_X{'F})(Y,\overline{Z})-(\tilde{B}_Y{'F})(X,\overline{Z})-(\tilde{B}_X{'F})(Y,\overline{Z})\nonumber\\&-&
(\tilde{B}_Y{'F})(\overline{Z},X)-(\tilde{B}_{\overline{Z}}{'F})(X,Y)\nonumber\\&=&
(\tilde{B}_{\overline{X}}{'F})(Y,Z)-(\tilde{B}_{\overline{Y}}{'F})(Y,Z)-(\tilde{B}_{\overline{Z}}{'F})(X,Y)
\end{eqnarray*}
Using (\ref{26}) in the above equation, we have 
\begin{eqnarray*}
'N(X,Y,Z)&-&d'F(X,Y,{\overline{Z}})=-A(X)[(\tilde{B}_{\overline{Y}}{A})(\overline{Z})+(\tilde{B}_{\overline{Z}}{A})(\overline{Y})]\nonumber\\&+&
A(Y)[(\tilde{B}_{\overline{X}}{A})(\overline{Z})+(\tilde{B}_{\overline{Z}}{A})(\overline{X})]+A(Z)[(\tilde{B}_{\overline{Y}}{A})(\overline{X})+(\tilde{B}_{\overline{X}}{A})(\overline{Y})]
\end{eqnarray*}
Since $U$ is killing , then 
\begin{equation*}
'N(X,Y,Z)-d'F(X,Y,{\overline{Z}})=2A(Z)(\tilde{B}_{\overline{Y}}{A})(\overline{X})
\end{equation*}
\end{proof}
\begin{Corollary}
If $'F$ is closed , then
\begin{equation*}
'N(X,Y,\overline{Z})=0
\end{equation*}
\end{Corollary}
\begin{Theorem}
A generalized co-symplectic manifold is quasi-Sasakian if
\begin{equation*}
(\tilde{B}_X{'F})(Y,T)=(\tilde{B}_Y{'F})(X,T),
\end{equation*}
where $\tilde{B}$ being a semi-symmetric non-metric connexion.
\end{Theorem}
\begin{proof}
From (\ref{22}) and (\ref{26}), we have
\begin{eqnarray*}
&&(D_X{'F})(Y,Z)+(D_Y{'F})(Z,X)+(D_Z{'F})(X,Y)\nonumber\\&=&
A(X)[(\tilde{B}_Z{A})(\overline{Y})-(\tilde{B}_Y{A})(\overline{Z})]
+A(Y)[(\tilde{B}_X{A})(\overline{Z})\nonumber\\&-&
(\tilde{B}_Z{A})(\overline{X})]+A(Z)[(\tilde{B}_Y{A})(\overline{X})-(\tilde{B}_X{A})(\overline{Y})]\nonumber\\&=&
A(X)[(\tilde{B}_Z{'F})(T,Y)-g(\overline{Y},\overline{Z})-(\tilde{B}_Y{'F})(T,Z)+g(\overline{Y},\overline{Z})]\nonumber\\&+&
A(Y)[(\tilde{B}_X{'F})(T,Z)-g(\overline{X},\overline{Z})-(\tilde{B}_Z{'F})(T,X)+g(\overline{X},\overline{Z})]\nonumber\\&+&
A(Z)[-(\tilde{B}_Y{'F})(X,T)-g(\overline{X},\overline{Y})+(\tilde{B}_X{'F})(Y,T)+g(\overline{X},\overline{Y})]\nonumber\\&=&
A(X)[(\tilde{B}_Z{'F})(T,Y)-(\tilde{B}_Y{'F})(T,Z)]+A(Y)[(\tilde{B}_X{'F})(T,Z)\nonumber\\&-&
(\tilde{B}_Z{'F})(T,X)]+A(Z)[(\tilde{B}_X{'F})(Y,T)-(\tilde{B}_Y{'F})(X,T)]\nonumber\\&=&0,
\end{eqnarray*}
\end{proof}
which proved the statement.
\section{Covariant almost analytic vector field}
If 1-form $w$ satisfies
\begin{equation}
\label{32}
w((D_X{F})(Y)-(D_Y{F})(X))=(D_{\overline{X}}{w})(Y)-(D_X{w})(\overline{Y})
\end{equation}
then 1-form $w$ is said to be covariant almost analytic vector field \cite{6}. Here $D$ is the Riemannian connection.
\begin{Theorem}
On an almost contact metric manifold if 1-form $w$ is covariant almost analytic vector field with respect to the Riemannian connexion $D$, then the semi-symmetric non-metric connexion $\tilde{B}$ coincide with $D$ if $w(\rho)=0$.
\end{Theorem}
\begin{proof}
We have ,
\begin{equation}
\label{33}
(\tilde{B}_X{F})(Y)=(D_X{F})(Y)+g(\overline{X},\overline{Y})\rho-g(\overline{X},Y){\overline{\rho}}
\end{equation}
Interchanging $X$ and $Y$ , we have
\begin{equation}
\label{34}
(\tilde{B}_Y{F})(X)=(D_Y{F})(X)+g(\overline{Y},\overline{X})\rho-g(\overline{Y},X){\overline{\rho}}
\end{equation}
Subtracting (\ref{34}) from (\ref{33}) and using (\ref{4}) and (\ref{2}), we obtain
\begin{eqnarray}
\label{35}
w((D_X{F})(Y)&-&(D_Y{F})(X))=w((\tilde{B}_X{F})(Y)\nonumber\\&-&(\tilde{B}_Y{F})(X))+2g(\overline{X},Y)w({\overline{\rho}})
\end{eqnarray}
Again,
\begin{equation}
\label{36}
(\tilde{B}_{\overline{X}}{w})(Y)=(D_{\overline{X}}{w})(Y)-w(\rho)g(\overline{\overline{X}},Y)
\end{equation}
and
\begin{equation}
\label{37}
(\tilde{B}_X{w})(\overline{Y})=(D_X{w})({\overline{Y}})-w(\rho)g(\overline{X},\overline{Y})
\end{equation}
From last two expressions, we get
\begin{eqnarray}
\label{38}
(D_{\overline{X}}{w})(Y)-(D_X{w})({\overline{Y}})&=&(\tilde{B}_{\overline{X}}{w})(Y)-(\tilde{B}_X{w})(\overline{Y})\nonumber\\&+&
w(\rho)[g(\overline{\overline{X}},Y)-g(\overline{X},\overline{Y})]
\end{eqnarray}
Now subtracting (\ref{38}) from (\ref{35}) and using (\ref{32}), we obtain that $w$ is covariant almost analytic vector field with respect to semi-symmetric non-metric connexion $\tilde{B}$ if and only if
\begin{equation*}
w(\rho)g(\overline{X},\overline{Y})=0
\end{equation*}
But in general, $g(\overline{X},\overline{Y})=0$ is not possible, therefore $w(\rho)=0$.
\end{proof}
\begin{Theorem}
On an almost contact metric manifold if 1-form $A$ is covariant almost analytic vector field with respect to the connexion $D$, then 
\begin{equation}
\label{40}
\tilde{d}A(X,Y)=dA(X,Y)+2g(X,\overline{Y}),
\end{equation}
where
\begin{equation}
\label{41}
dA(X,Y){\stackrel{\mathrm{def}}{=}}(D_X{A})(Y)-(D_Y{A})(X)
\end{equation}
\begin{equation}
\label{42}
\tilde{d}A(X,Y){\stackrel{\mathrm{def}}{=}}(\tilde{B}_X{A})(Y)-(\tilde{B}_Y{A})(X)
\end{equation}
\end{Theorem}
\begin{proof}
Interchanging $X$ and $Y$ in (\ref{15}), we get
\begin{equation}
\label{43}
(\tilde{B}_Y{A})(X)=(D_Y{A})(X)-g(X,\overline{Y})
\end{equation}
From (\ref{15}) and (\ref{43}), we obtain
\begin{equation}
\label{44}
(\tilde{B}_X{A})(Y)-(\tilde{B}_Y{A})(X)=(D_X{A})(Y)-(D_Y{A})(X)+2g(X,\overline{Y})
\end{equation}
In view of (\ref{41}), (\ref{42}), (\ref{44}) gives (\ref{40}).
\end{proof}
\subsection*{Acknowledgement} The author wishes to express heartly thanks to Prof. R. H. Ojha for his kind guidences. 

\thebibliography{00}
\bibitem{1}
Hayden, H. A. (1932),Proc. London Math. Soc. 34 , 27-50.
\bibitem{2}
Nirmala, S., Agashe and Mangala R.Chafle (1992): A semi-symmetric non-metric connection on a Riemannian manifold, Indian J.pure appl. Math., 23(6), 399-409. 
\bibitem{3}
Pandey, P. N. and Dubey, S. K. (2004): Almost Grayan manifold admitting semi-symmetric metric connection, Tensor, N. S., Vol. 65, pp 144-152.
\bibitem{4}
Pandey, S. N. and Ojha, R. H. (1993): On generalized co-symplectic manifold, Instanbul Univ. Fen, Fak. Mat. Der. 52 , 17-21.
\bibitem{5}
Mishra, R. S. (1991): Almost contact metric manifolds, Monograph No. 1, Tensor Society of India, Lucknow.
\bibitem{6}
Mishra, R. S. (1984): Structures on a differentiable manifold and their applications, chandrama prakashan, Allahabad.
\bibitem{7}
Yano, K. (1970): On semi-symmetric metric connexion, Rev. Roumanino de Math. Pure Et Appliques, Vol. 15, pp.1579-1586.
\end{document}